\documentclass[a4paper,11pt,reqno]{amsart}

\usepackage{amsthm,amsmath,amsfonts,amssymb}

\usepackage{graphicx}
\usepackage{tikz}
\usetikzlibrary{decorations.pathreplacing}
\usepackage{subcaption}
\usepackage{color}

\usepackage[colorlinks=true,
urlcolor=teal,
citecolor=magenta,
linkcolor=teal,
linktocpage,
pdfpagelabels,
bookmarksnumbered,
bookmarksopen]{hyperref}
\newtheorem{thm}{Theorem}[section]

\newtheorem{lem}[thm]{Lemma}

\newtheorem{rem}[thm]{Remark}

\numberwithin{equation}{section}
\DeclareMathOperator{\erfc}{erfc}
\DeclareMathOperator{\sgn}{sgn}

\newcommand{\R}{{\mathbb{R}}}
\newcommand{\N}{{\mathbb{N}}}
\newcommand{\e}{{\varepsilon}}
\renewcommand{\epsilon}{{\varepsilon}}
\renewcommand{\t}{{\theta}}
\newcommand{\g}{{\gamma}}

\newcommand{\de}{{\partial}}

\begin{document}

\title[Quantitative lower bounds to the Cheeger constant]{Quantitative lower bounds to\\the Euclidean and the Gaussian\\ Cheeger constants}
\author{Vesa Julin}
\address{Jyv\"askyl\"an yliopisto, Department of Mathematics and Statistics, P.O.Box 35 (MaD), FI-40014 Jyv\"askyl\"a}
\email{vesa.julin@jyu.fi}
\author{Giorgio Saracco}
\address{Scuola Internazionale Superiore di Studi Avanzati (SISSA), via Bonomea 265, IT--34136 Trieste}
\email{gsaracco@sissa.it}
%
%

%\thanks{The present research was carried out during a visit of G.~S.~at the University of Jyv\"askyl\"a.  G.~S.~wishes to thank the hosting institution for the kind hospitality and the INdAM institute of which he is a member and which funded his stay in Jyv\"askyl\"a (grant n.~U-UFMBAZ-2018-000928 03-08-2018). G.~S.~was also partially supported by the INdAM--GNAMPA 2019 project ``Problemi isoperimetrici in spazi Euclidei e non'' (grant n.~U-UFMBAZ-2019-000473 11-03-2019). {\color{blue}V.~J.~was supported by the Academy of Finland grant 314227.}}



\begin{abstract}
We provide a quantitative lower bound to the Cheeger constant of a set $\Omega$ in both the Euclidean and the Gaussian settings in terms of suitable asymmetry indexes. We provide examples which show that these quantitative estimates are sharp.
\end{abstract}
 \hspace{-2cm}
 {
 \begin{minipage}[t]{0.6\linewidth}
 \begin{scriptsize}
 \vspace{-3cm}
 This is a pre-print of an article published in \emph{Ann. Fenn. Math.}. The final authenticated version is available online at: \href{https://doi.org/10.5186/aasfm.2021.4666}{https://doi.org/10.5186/aasfm.2021.4666}
 \end{scriptsize}
\end{minipage} 
}

\maketitle

\section{Introduction}

In the past years inequalities of geometric-functional type have been widely studied in the literature, a --- far from complete  ---  list is~\cite{FMP10, Fug89, FJ14, FMP08} (isoperimetric inequalities),~\cite{EFT05, Neu16} (anisotropic Wulff inequalities),~\cite{BBJ17, BJ17, CFMP11} (Gaussian inequalities),~\cite{FLu, FP19, HN94} (Riesz inequalities),~\cite{Cia06, Cia08, CFMP09, FMP07, Neu20} (Sobolev inequalities),~\cite{BdPV15, Fra08, FMP09} (Faber--Krahn inequalities, see also~\cite{BdP17} for an account on other quantitative spectral inequalities).\par

In this paper we are interested in providing quantitative estimates on the \emph{Cheeger constant} both in the Euclidean and in the Gaussian setting. Given any open set $\Omega$, of finite, resp.~Euclidean or Gaussian, measure, the constant is defined as, resp.,
\begin{equation}\label{eq:Cc}
h(\Omega):= \inf \left\{\,\frac{P(E)}{|E|}\, \right\},\qquad h_\gamma(\Omega):= \inf \left\{\,\frac{P_\gamma(E)}{\gamma(E)}\right\},
\end{equation}
where the infima are taken among subsets $E\subset \Omega$ of positive,  resp.~Euclidean or Gaussian, measure. In~\eqref{eq:Cc}, we denote by $P(\cdot)$ and $P_\gamma(\cdot)$, resp., the distributional Euclidean and Gaussian perimeter, while by $|\cdot|$ and $\gamma(\cdot)$, resp., the standard Lebesgue and Gaussian measure.\par

Sets attaining the above infima are known to exist and are called \emph{Cheeger sets}, see for instance~\cite{Leo15, Par11} for the Euclidean case, and~\cite{CMN10} for the Gaussian case (more general settings have been explored, see for instance~\cite{BLP14, LT19, Sar18}). The task of computing the constant and determining the shape of Cheeger sets is usually referred to as the \emph{Cheeger problem}. The constant first appeared in~\cite{Che70} in a Riemannian setting as a mean to bound from below the first Dirichlet eigenvalue of the Laplace--Beltrami operator. Through the coarea formula it can be proven that the Euclidean constant $h(\Omega)$ provides a lower bound to the first Dirichlet eigenvalue of the Laplace operator, and analogously the Gaussian constant $h_\gamma(\Omega)$ to the first Dirichlet eigenvalue of the Ornstein--Uhlenbeck operator (i.e.,~$-\Delta(\cdot) + \langle x, \nabla (\cdot) \rangle$).\par 

Since then, the Euclidean problem has been widely studied and it has  appeared in many different contexts, such as capillarity problems~\cite{Giu78, LS18a}, spectral properties of the $p$-Laplacian~\cite{KF03}, and landslide modeling~\cite{ILR05} to name a few. The interested reader is referred to the surveys~\cite{Leo15, Par11} and the references therein. The Gaussian counterpart was studied in relation to the prescribed mean curvature problem~\cite{CMN10} and to image processing~\cite{CFM09}.\par

It is of interest to provide estimates on the constants as there are no available formulas to directly compute them, but for very few classes of sets $\Omega$ limited to the Euclidean $2$-dimensional setting, see~\cite{KL06, LNS17, LP16, LS19, Sar19}. In higher dimensions, some very special cases are treated in~\cite{BP19, BE11, Dem04, KLV19}. To give an upper bound to $h(\Omega)$ (resp., $h_\gamma(\Omega)$) it is enough to compute the ratio $P(E)/|E|$ (resp., $P_\gamma(E)/\gamma(E)$) for any competitor $E$, while establishing lower bounds exploits the relevant isoperimetric inequalities. Indeed, first, rough estimates on the constants are provided by
\begin{equation}\label{eq:est_iso_in}
h(\Omega) \ge h(B_\Omega), \qquad h_\gamma(\Omega) \ge h_\gamma(H_\Omega),
\end{equation}
where by $B_\Omega$ we denote the ball centered at the origin s.t.~$|\Omega|=|B_\Omega|$, and by $H_\Omega$ any halfspace s.t.~$\gamma(\Omega)=\gamma(H_\Omega)$.\par

The isoperimetric inequalities have been proven in quantitative forms by establishing bounds through asymmetry indexes measuring the distance (in some suitable sense) of $\Omega$ from the isoperimetric set with the same measure. Then, it is natural to wonder if any improvement to~\eqref{eq:est_iso_in} can be attained by exploiting these stronger inequalities and get lower bounds of the form
\begin{equation}\label{eq:obj}
\frac{h(\Omega)-h(B_\Omega)}{h(B_\Omega)} \ge c(n)\iota(\Omega)
\end{equation}
in the Euclidean case, and
\begin{equation}\label{eq:obj2}
h_\gamma(\Omega) - h_\gamma(H_\Omega) \ge c(\gamma(\Omega))\iota_\gamma(\Omega)
\end{equation}
in the Gaussian case, where $\iota(\Omega)$ and $\iota_\gamma(\Omega)$ are some suitable asymmetry indexes. These inequalities  give an improved lower bound on the Cheeger constant for sets that are near the corresponding isoperimetric set.

We remark that there are two main differences in the Euclidean ~\eqref{eq:obj} and the Gaussian~\eqref{eq:obj2} inequalities that we expect to prove. First, in the Euclidean case, the quantitative estimate is renormalized, while in the Gaussian it is not. This is so because the respective quantitative isoperimetric estimates are renormalized in the former setting and not in the latter. Second, the constant $c$ in the Euclidean case depends only on the dimension $n$, while in the Gaussian only on the measure of $\Omega$. Again, this is a known feature of the quantitative isoperimetric inequalities in the two different settings.\par

A first result in this direction has been proven in the Euclidean case in~\cite{FMP09} with $\iota(\Omega)=\alpha^3(\Omega)$, where $\alpha$ is known as the Fraenkel asymmetry index. This was later refined in~\cite{FMP09a}, with the exponent $2$ in place of $3$, i.e., $\iota(\Omega)=\alpha^2(\Omega)$, and this exponent is known to be sharp, i.e., no such inequality can hold with a smaller power of $\alpha$. Our first main result, Theorem~\ref{thm:zeta}, states  that in the Euclidean case a stronger quantitative inequality holds, where the index is given by the Riesz asymmetry index. To the best of our knowledge, there are no previous results in the Gaussian case. In Theorem~\ref{thm:gauss_alpha} we prove a sharp quantitative inequality of type  \eqref{eq:obj2} in terms of the Gaussian Fraenkel asymmetry. Then we consider a different index in terms of the barycenter, which was  introduced in~\cite{BBJ17, Eld15}, and prove a sharp quantitative inequality in terms of this in Theorem~\ref{thm:gauss_logbeta}. Rather surprisingly the optimal dependence on the asymmetry in  Theorem~\ref{thm:gauss_logbeta} is different than in the quantitative Gaussian isoperimetric inequality of~\cite{BBJ17} and it has a logarithmic dependence on the barycenter index as in~\cite{Eld15}. \par 

The paper is organized into two independent sections.  In Section~\ref{sec:Eucl} we study  the Euclidean case \eqref{eq:obj} and  in Section~\ref{sec:Gaus} the Gaussian one \eqref{eq:obj2}. Each section is self-contained and begins with relevant definitions, related isoperimetric inequalities and statements of the results. The proof of the theorems follows. Finally, each section is complemented with an example that shows that the  quantitative inequalities are sharp. 

\section{Estimates in the Euclidean setting}\label{sec:Eucl}

In the Euclidean setting there are three quantitative isoperimetric inequalities available, in terms of the following three indexes:
\begin{align}
\alpha(\Omega)&:=\min_{y\in \R^n}\left\{\,\frac{|\Omega \Delta (B_\Omega+y)|}{|\Omega|}\, \right\}, \label{eq:def_alpha}\\
\zeta(\Omega) &:= \frac{1}{P(B_\Omega)}\left (\int_{B_\Omega}\frac{\textrm{d}x}{|x|} - \max_{y\in \R^n}\int_\Omega \frac{\textrm{d}x}{|x-y|} \right)\label{eq:def_zeta}\\
\beta(\Omega)&:= \min_{y\in \R^n} \left\{\,  \left( \frac{1}{2P(B_\Omega)}\int_{\de^* \Omega}|\nu_\Omega(x) - \nu_{B_\Omega+y}(\pi_{y,\Omega}(x))|^2  \, \textrm{d}\mathcal{H}^{n-1}(x)\right)^{\frac 12}\, \right\},\label{eq:def_beta}
\end{align}
where  $\pi_{y,\Omega}(x)$ is the projection of $\R^n\setminus \{y\}$ on the boundary of $B_\Omega+y$, i.e.,
\begin{align*}
\pi_{y,\Omega}(x) := y + r\frac{x-y}{|x-y|}\,, && \forall x\neq y,
\end{align*}
being $r$ the radius of $B_\Omega$. Indeed, one has that there exists a constant $c=c(n)$ depending only on the dimension (which changes from line to line) such that
\begin{align}
\frac{P(\Omega)-P(B_\Omega)}{P(B_\Omega)}&\ge c\, \alpha^2(\Omega),\label{eq:iso_ine_alpha}\\
\frac{P(\Omega)-P(B_\Omega)}{P(B_\Omega)}&\ge c\, \zeta(\Omega),\label{eq:iso_ine_zeta}\\
\frac{P(\Omega)-P(B_\Omega)}{P(B_\Omega)}&\ge c\, \beta^2(\Omega).\label{eq:iso_ine_beta}
\end{align}
Inequality~\eqref{eq:iso_ine_alpha} was proven in~\cite{FMP10, FMP08}, while inequalities\footnote{We remark that inequality~\eqref{eq:iso_ine_zeta} is not explicitly stated, but it is contained in the proof of~\cite[Proposition~1.2]{FJ14}.}~\eqref{eq:iso_ine_zeta} and~~\eqref{eq:iso_ine_beta} in~\cite{FJ14}, and the exponents are known to be sharp. The interested reader is referred to the beautiful survey~\cite{Fus15}. Inequalities~\eqref{eq:iso_ine_zeta} and~\eqref{eq:iso_ine_beta} are subsequent refinements of~\eqref{eq:iso_ine_alpha}, in the sense that the indexes $\beta$ and $\zeta$ are better than $\alpha$, i.e.,~one has 
\[
\beta^2(\Omega) \gtrsim \zeta(\Omega) \gtrsim \alpha^2(\Omega)
\]
for any set of finite perimeter $\Omega$.

Inequality~\eqref{eq:iso_ine_alpha} has been successfully used to give a quantitative estimate on the Cheeger constant in~\cite{FMP09a} in terms of the Fraenkel asymmetry $\alpha$, defined above in~\eqref{eq:def_alpha}.  It is then natural to wonder whether the inequalities~\eqref{eq:iso_ine_zeta} and~\eqref{eq:iso_ine_beta} have an analogous counterpart in terms of the Cheeger constant. Our first main result states that  this is indeed true for the index $\zeta$.

\begin{thm}\label{thm:zeta}
Let $\Omega$ be an open, bounded set in $\R^n$. There exists a dimensional constant $c=c(n)$ such that
\begin{equation}\label{eq:che_ine_zeta}
\frac{h(\Omega) - h(B_\Omega)}{h(B_\Omega)} \ge  c\,\zeta(\Omega),
\end{equation}
where $\zeta(\cdot)$ is defined in~\eqref{eq:def_zeta}. 
\end{thm}

We remark that~\eqref{eq:che_ine_zeta} is sharp, and this follows because~\eqref{eq:iso_ine_zeta} is sharp. In Section~\ref{ssec:beta_not} we give an example that shows that  a quantitative inequality analogous to~\eqref{eq:che_ine_zeta} does not hold with the index $\beta^2$ in place of the index $\zeta$. We remark that the index $\zeta$  is used, e.g.,~in~\cite{Jul14} to prove the minimality of the ball in Gamow's liquid drop model for small masses.   \par

We also remark that a similar analysis can be carried over when considering the $m$-Cheeger sets studied in~\cite{PS17} wherein the ratio defining $h(\Omega)$ one considers suitable powers $m$ of the volume rather than the power $1$.

\subsection{Proof of Theorem~\ref{thm:zeta}}\label{ssec:zeta_ok}

First, thanks to the scaling property of the Euclidean Cheeger constant, i.e.,~$h(\lambda \Omega) = \lambda^{-1}h(\Omega)$ for $\lambda>0$, it is enough to prove the inequality for $\Omega$ s.t.~$|\Omega|=\omega_n$, i.e.,~$B_\Omega$ is the unit ball.\par

Second, notice that
\begin{equation}\label{eq:potential}
\int_{B_\Omega} \frac{\textrm{d}x}{|x|} = n\omega_n  \int_0^1 \frac{\rho^{n-1}}{\rho}\,\textrm{d}\rho = \frac{P(B_\Omega)}{n-1},
\end{equation}
thus $\zeta(\Omega)\le (n-1)^{-1}$. Therefore, the inequality immediately follows for sets $\Omega$ s.t.~$h(\Omega) \ge 2h(B_\Omega)$, by choosing $c(n)\le2$.\par

Hence, let us consider $\Omega$ with volume $\omega_n$ s.t.~$h(\Omega) < 2h(B_\Omega)$, and denote by $E$ a Cheeger set of $\Omega$. We begin by estimating $\zeta(\Omega)$ in terms of $\zeta(E)$. Up to translating $E$ (and therefore $\Omega$) we  may  assume that $E$ is ``centered''
\[
P(B_E)\zeta(E) = \int_{B_E}\frac{\textrm{d}x}{|x|} - \int_{E}\frac{\textrm{d}x}{|x|},
\]
i.e.,~the maximum in~\eqref{eq:def_zeta} is attained at  the origin $y=0$. By~\eqref{eq:potential}, adding and removing $(n-1)^{-1}P(B_E)$, using the positivity of the integrands and the fact that $E\subset \Omega$ to estimate $-\max_y \int_\Omega |x-y|^{-1}\mathrm{d}x$, dividing by $P(B_\Omega)$ and owing to the fact that $P(B_E)\le P(B_\Omega)$ we have
\begin{align}\label{eq:eq1}
\zeta(\Omega) 
&= \frac{1}{n-1} -\frac{1}{P(B_\Omega)}\max_{y\in \R^n}\int_\Omega \frac{\textrm{d}x}{|x-y|}
\nonumber
\\
&\le \frac{1}{n-1}\frac{P(B_\Omega)-P(B_E)}{P(B_\Omega)}  +\zeta(E).
\end{align}
We aim to bound both terms on the RHS through the renormalized difference of the Cheeger constants $(h(\Omega)-h(B_\Omega))h(B_\Omega)^{-1}$, up to some dimensional constant $c=c(n)$. This would conclude the proof.\par
To estimate the first term on the RHS  of \eqref{eq:eq1}, we exploit the isoperimetric inequality to deduce
\begin{equation*}
h(\Omega) = \frac{P(E)}{|E|} \ge \frac{P(B_E)}{|B_E|} = h(B_E) = \left(\frac{|\Omega|}{|E|} \right)^{\frac 1n}h(B_\Omega),
\end{equation*}
where the last equality follows from the scaling properties of $h(\cdot)$. By employing equality $P(B)=n\omega_n^{\frac 1n}|B|^{\frac{n-1}{n}}$ valid for any ball $B$, the above inequality, and recalling that $|\Omega| = \omega_n$, we obtain 
\begin{align*}
\frac{1}{n-1}\frac{P(B_\Omega)-P(B_E)}{P(B_\Omega)} 
&
=
\frac{1}{n-1}\left( 1 - \left( \frac{|E|}{|\Omega|}\right)^{\frac{n-1}{n}} \right)
\\
&
\le 
\frac{1}{n-1}\frac{h(\Omega)^{n-1}-h(B_\Omega)^{n-1}}{h(\Omega)^{n-1}}.
\end{align*}
As $t^a -s^a \le a t^{a-1}(t-s)$ whenever $a\ge 1$, and $s\in (0,t]$, and recalling that $h(\Omega) \ge h(B_\Omega)$ we finally get
\begin{equation}\label{eq:eq2}
\frac{1}{n-1}\frac{P(B_\Omega)-P(B_E)}{{P(B_\Omega)}} \le \frac{h(\Omega) - h(B_\Omega)}{h(\Omega)}.
\end{equation}

We are left with providing an estimate to $\zeta(E)$. In order to do so, recall that $|\Omega|=|B_\Omega|=\omega_n$, and $h(B_\Omega)=n$. Thus, the following chain of equalities holds
\begin{align*}
|B_\Omega|^{-\frac 1n}\frac{h(\Omega)-h(B_\Omega)}{h(B_\Omega)} &= \frac{1}{n |B_\Omega|^{\frac 1n}}\frac{P(E)}{|E|} -  |B_\Omega|^{-\frac 1n}\\
&= \frac{P(E)}{n |B_\Omega|^{\frac 1n}|E|^{\frac{n-1}{n}}}|E|^{-\frac 1n} - |B_\Omega|^{-\frac 1n}\\
&=\frac{P(E)-P(B_E)}{P(B_E)}  |E|^{-\frac 1n} + \left( |E|^{-\frac 1n} - |B_\Omega|^{-\frac 1n}\right )\\
&\ge \frac{P(E)-P(B_E)}{P(B_E)}  |E|^{-\frac 1n},
\end{align*}
where the last inequality follows from $E\subset \Omega$. Therefore by using~\eqref{eq:iso_ine_zeta} and the above inequality we get
\begin{equation}\label{eq:eq3}
\zeta(E) \le c(n)  \frac{P(E)-P(B_E)}{P(B_E)} \le  c(n)\frac{h(\Omega)-h(B_\Omega)}{h(B_\Omega)} \left(\frac{|E|}{|B_\Omega|} \right)^{\frac 1n}.
\end{equation}
As $|E|\le |B_\Omega|$, combining~\eqref{eq:eq1} with~\eqref{eq:eq2} and~\eqref{eq:eq3} yields the  claim.

\subsection{Failure of the inequality with the index \texorpdfstring{$\beta^2$}{beta^2}}\label{ssec:beta_not}

In this section, we show that there is no constant $c=c(n)$ such that inequality
\begin{equation}\label{eq:failure_beta_eucl}
\frac{h(\Omega) - h(B_\Omega)}{h(B_\Omega)}\ge c\, \beta^2(\Omega)
\end{equation}
holds true, by building a suitable family of sets. Let us fix $\e<<1$ and consider the family of bounded sets $\{\Omega_j\}_{j\in \N}$, where the boundary of $\Omega_j$ is given by the closed, simple polar curve
\begin{align*}
f_j(\t) = \left(1 +\frac{\e^2}{2}\right)^{-\frac 12} (1+\e \sin(2j\t)), && \t\in [0, 2\pi],
\end{align*}
some of which are depicted in Figure~\ref{fig:Omega_n}. 
 \begin{figure}[t]
    \centering
%%%%FIG.1
    \begin{subfigure}{0.32\linewidth}
    \centering
       		\includegraphics[width=1\linewidth]{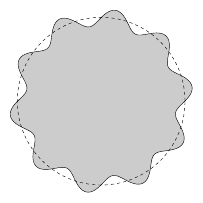}
               \caption{$j=5$}
        \label{fig:a}
    \end{subfigure}
%%%%FIG.2       
    \begin{subfigure}{0.32\linewidth}
    \centering
       	\includegraphics[width=1\linewidth]{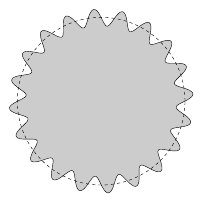}
        \caption{$j=10$}
        \label{fig:b}
    \end{subfigure}
%%%%FIG.3
   \begin{subfigure}{0.32\linewidth}
   \centering
       \includegraphics[width=1\linewidth]{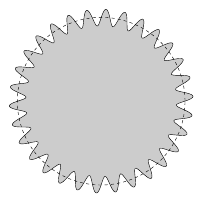}
        \caption{$j=15$}
        \label{fig:c}
    \end{subfigure}
    \caption{The set $\Omega_j$, for different values of $j$, with the choice $\varepsilon=10^{-1}$. The dashed line represents the unit ball.}
    \label{fig:Omega_n}
    \end{figure}
The volume of $\Omega_j$ is
\begin{align*}
|\Omega_j| &= \frac 12 \int_0^{2\pi} f^2_j(\t)\, \textrm{d}\t \\
 &= \frac 12 \left(1 +\frac{\e^2}{2}\right)^{-1} \int_0^{2\pi} 1+2\e \sin(2j\t) + \e^2 \sin^2(2j\t)\,\textrm{d}\t \\
 &=  \frac 12 \left(1 +\frac{\e^2}{2}\right)^{-1} (2\pi + \e^2 \pi) = \pi.
\end{align*}
The perimeter of $\Omega_j$ can be easily estimated from below as
\begin{align*}
P(\Omega_j) &=  \int_0^{2\pi} \sqrt{f^2_j(\t)+(f^\prime_j(\t))^2}\, \textrm{d}\t \\
 &= \left(1 +\frac{\e^2}{2}\right)^{- \frac 12} \int_0^{2\pi} \sqrt{(1+ \e \sin(2j\t))^2 + (2j\e \cos(2j\t))^2}\,\textrm{d}\t \\
 &\ge  2j\e \left(1 +\frac{\e^2}{2}\right)^{- \frac 12} \int_0^{2\pi} | \cos(2j\t)|\,\textrm{d}\t  = 8 j\e \left(1 +\frac{\e^2}{2}\right)^{- \frac 12}.
\end{align*}
As we let $j \to \infty$, we see that $P(\Omega_j) \to \infty$. Thus, all the sets $\Omega_j$ have the volume of the unit ball $B_1$, while their perimeter diverges.\par

Moreover, the sets $\Omega_j$ contain the ball $B_{\min_\t f_j(\t)}$, where such a radius is readily computed to be
\[
\min_\t f_j(\t) = \left(1 +\frac{\e^2}{2}\right)^{-\frac 12} (1-\e).
\]
Therefore $h(B_{\min_\t f_j(\t)})$ provides an upper bound to $h(\Omega_j)$ and  $h(B_1)= 2$ a lower bound, i.e., 
\[
\frac{2}{1-\e} \left(1 +\frac{\e^2}{2}\right)^{\frac 12} \ge h(\Omega_j) \ge 2.
\]
Note that, by recalling the definition~\eqref{eq:def_beta} of $\beta$, we may write  $\beta^2(\Omega)$ as 
\begin{align*}
P(B_\Omega)\beta^2(\Omega) &= \frac{1}{2}\min_{y\in \R^n} \int_{\de^* \Omega} 2 - 2 \nu_\Omega(x) \cdot  \nu_{B_\Omega(y)}(\pi_{y,\Omega}(x)) \, \textrm{d}\mathcal{H}^{n-1}(x)  \\
&= P(\Omega) - \max_{y\in \R^n} \int_{\de^* \Omega} \nu_\Omega(x) \cdot \frac{x-y}{|x-y|} \, \textrm{d}\mathcal{H}^{n-1}(x)  \\
&= P(\Omega) - (n-1)\,\max_{y\in \R^n} \int_\Omega \frac{\textrm{d}x}{|x-y|}   \\
&=  P(\Omega)-P(B_\Omega) + (n-1) P(B_\Omega)\zeta(\Omega).
\end{align*}
Thus, the inequality~\eqref{eq:failure_beta_eucl} cannot hold for all $j$ for any choice of $c(n)$, as $h(\Omega_j)$ is uniformly bounded from above and the oscillation index
\[
\beta^2(\Omega_j) \ge (P(\Omega_j)-P(B_1))\cdot P(B_1)^{-1}
\]
diverges as $j\to +\infty$.

\begin{rem}\label{rem:annulus}
We may construct easier examples, if we do not require the competing sets to be starshaped. Consider now the family of bounded sets $\{\Omega_j\}_{j \in \mathbb{N}}$ with
\[
\Omega_j := B_{1-\frac 1j} \cup A_{1, 1+\e(j)},
\]
where $A_{1, 1+\e(j)}$ is the annulus centered at the origin with inner radius $1$ and outer radius $1+\e(j)$, with $\e(j)$ such that $|\Omega_j|=|B_1|$, for all $j$. One of these sets is depicted in Figure~\ref{fig:annulus_a}. It is immediate to check that $h(\Omega_j)\to h(B_1)$, while at the same time $\beta^2(\Omega_j)\ge 2$. Thus the LHS of~\eqref{eq:failure_beta_eucl} goes to zero, while the RHS is uniformly strictly greater than zero. Notice that the sets of this family can be easily modified to ensure that they  are all connected, see Figure~\ref{fig:annulus_b}.
\end{rem}

\begin{figure}[t]

    \centering
%%%%FIG.1
    \begin{subfigure}{0.48\linewidth}
    \centering
       		\begin{tikzpicture}
\filldraw[black!20!white] (0,0) circle (2);
\fill[white] (0,0) circle (1.85);
\fill[black!20!white] (0,0) circle (1.65);
\draw[black] (0,0) circle (1.65);
\draw[black] (0,0) circle (1.85);
\draw[black] (0,0) circle (2);
	\end{tikzpicture}
               \caption{The non-connected set}
        \label{fig:annulus_a}
    \end{subfigure}
%%%%FIG.2       
    \begin{subfigure}{0.48\linewidth}
    \centering
       	\begin{tikzpicture}
\filldraw[black!20!white] (0,0) circle (1.95);
\fill[white] (0,0) circle (1.85);
\fill[black!20!white] (0,0) circle (1.65);
\draw[black] (0,0) circle (1.65);
\draw[black] (0,0) circle (1.85);
\draw[black] (0,0) circle (1.95);
\fill[black!20!white] (-0.1, 1.63) rectangle (0.1, 1.87);
\draw[black] (-0.1, 1.64) -- (-0.1, 1.8542);
\draw[black] (0.1, 1.64) -- (0.1, 1.8542);	
\fill[black!20!white] (-0.1, -1.63) rectangle (0.1, -1.87);
\draw[black] (-0.1, -1.64) -- (-0.1,-1.8542);
\draw[black] (0.1, -1.64) -- (0.1, -1.8542);
	\end{tikzpicture}
        \caption{The connected set}
        \label{fig:annulus_b}
    \end{subfigure}	
    \caption{The grayed-out sets represent one of the $\Omega_j$ introduced in Remark~\ref{rem:annulus}, and its connected, symmetric counterpart.\label{fig:annulus}}
\end{figure}
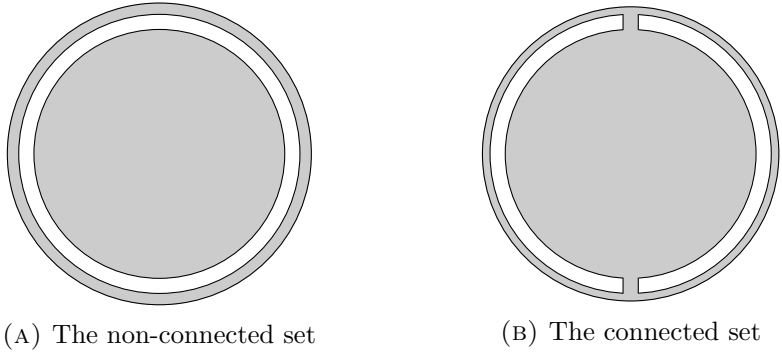

\section{Estimates in the Gaussian setting}\label{sec:Gaus}

Given a set of locally finite perimeter $E\subset \mathbb{R}^n$, we define  its Gaussian perimeter and volume to be
\[
P_\gamma(E) :=  \frac{1}{(2\pi)^{\frac {n-1}{2}}}\int_{\partial^* E} e^{-\frac{|x|^2}{2}}\, \textrm{d}\mathcal{H}^{n-1}(x), \qquad \gamma(E):=  \frac{1}{(2\pi)^{\frac n2}}\int_{E} e^{-\frac{|x|^2}{2}}\, \textrm{d}x.
\]
Given any direction $\omega \in \mathbb{S}^{n-1}$ and any real number $s\in \mathbb{R}$ we denote by $H_{s,\omega}$ the halfspace $\{\,x\in \mathbb{R}^n: x\cdot \omega < s\,\}$. We denote by $\phi$ the function
\[
\phi(s):= \frac{1}{\sqrt{2\pi}}\int_{-\infty}^s e^{-\frac{t^2}{2}}\, \textrm{d}t,
\]
and have that for any direction $\omega \in \mathbb{S}^{n-1}$ it holds
\[
P_\gamma(H_{s, \omega}) = e^{-\frac{s^2}{2}}, \qquad \gamma(H_{s, \omega})=\phi(s).
\]
If the direction $\omega$ is not relevant, we shall drop it and write $H_s$ as a shorthand for any halfspace of measure $\phi(s)$. Moreover, given any set $E$, we denote by $H_E$ any halfspace such that $\gamma(E)=\gamma(H_E)$.  If the direction is relevant  we denote it by $H_{E, \omega}$. The Gaussian isoperimetric inequality states that
\begin{equation}\label{eq:gauss_iso}
P_\gamma(E) \ge P_\gamma(H_E),
\end{equation}
with equality if and only if $E$ is a halfspace, see for instance~\cite{Bor75, CK01, SC74}. Analogously to the Euclidean case, quantitative versions of~\eqref{eq:gauss_iso} have been proven, namely there exists a positive constant $c=c(\gamma(E))$ depending only on the measure of $E$ (which changes from line to line) such that
\begin{align}
P_\gamma(E) - P_\gamma(H_E) \ge c\, \alpha_\gamma^2(E), \label{eq:gauss_iso_alpha}\\
P_\gamma(E) - P_\gamma(H_E) \ge c\, \beta_\gamma(E),  \label{eq:gauss_iso_beta}
\end{align}
where the indexes $\alpha_\gamma$ and $\beta_\gamma$ are given by
\begin{align}
\alpha_\gamma(E) &:= \min_{\omega\in \mathbb{S}^{n-1}} |E\Delta H_{E,\omega}|, \label{eq:def_alpha_gauss}\\ 
\beta_\gamma(E) &:= \min_{\omega\in \mathbb{S}^{n-1}} |b(H_{E,\omega}) - b(E)|,\label{eq:def_beta_gauss}
\end{align}
where $b(E)$ is the non-renormalized barycenter of $E$, i.e.,
\[
b(E) := \frac{1}{(2\pi)^\frac{n}{2}}\int_E x e^{-\frac{|x|^2}{2}}\mathrm{d}x.
\]
It is easy to see  that $|b(E)|$ is maximized by the halfspace $H_E$,  i.e., $|b(E)| \leq |b(H_E)|$,  and $s \mapsto |b(H_s)|$ attains its  maximum  at  $s=0$  with $|b(H_0)|=(2\pi)^{-\frac 12}$. Moreover we note that  $\beta_\gamma(E) \le 1$ for every set $E$. For an account of these facts, we refer the reader to~\cite{BBJ17}, where also the inequalities \eqref{eq:gauss_iso_alpha} and \eqref{eq:gauss_iso_beta} are proven (see also \cite{CFMP11, Eld15}). As in the Euclidean case the index $\beta_\gamma$ is stronger than $\alpha_\gamma$, in the sense that 
\[
\beta_\gamma(E) \gtrsim \alpha_\gamma^2(E)
\]
for every measurable set $E\subset \R^n$.  \par

The two following theorems are the main results of this section and they are proven respectively in Section~\ref{ssec:gauss_alpha} and Section~\ref{ssec:gauss_logbeta}.

\begin{thm}
\label{thm:gauss_alpha}
Let $\Omega$ be an open set in $\R^n$. There exists a constant $c=c(\gamma(\Omega))$ such that
\begin{equation}\label{eq:che_gaus_alpha}
h_\gamma(\Omega) - h_\gamma(H_\Omega) \geq c \, \alpha_\gamma^2(\Omega),
\end{equation}
where $\alpha_\gamma(\cdot)$ is defined in~\eqref{eq:def_alpha_gauss}.
\end{thm}

\begin{thm}
\label{thm:gauss_logbeta}
Let $\Omega$ be an open set in $\R^n$. There exists a constant $c=c(\gamma(\Omega))$ such that
\begin{equation}\label{eq:che_gaus_beta}
h_\gamma(\Omega) - h_\gamma(H_\Omega) \geq c \, \frac{\beta_\gamma(\Omega)}{1 + \sqrt{|\log(\beta_\gamma(\Omega))|}},
\end{equation}
where $\beta_\gamma(\cdot)$ is defined in~\eqref{eq:def_beta_gauss}.
\end{thm}

We remark that neither inequality~\eqref{eq:che_gaus_beta} implies inequality~\eqref{eq:che_gaus_alpha}, nor is inequality~\eqref{eq:che_gaus_alpha} stronger than inequality~\eqref{eq:che_gaus_beta}.
Finally, we show in Section~\ref{ssec:gaus_beta_not} by means of an example that the dependence on the asymmetry in~\eqref{eq:che_gaus_beta} is optimal (see also the result in~\cite{Eld15}). \par

\subsection{Preliminary lemmas}

In this section, we prove some lemmas regarding properties of one-dimensional functions, which are useful in the proofs of Theorems~\ref{thm:gauss_alpha} and~\ref{thm:gauss_logbeta}. We recall the definition of the complementary error function $\erfc(\cdot)$ and some lower and upper bounds to it, which we will use later. Given $x>0$, we set 
\[
\erfc(x) := \frac{2}{\sqrt{\pi}} \int_{x}^{+\infty} e^{t^2}\mathrm{d}t.
\]
For $x>>1$ one has
\begin{equation}\label{eq:Taylor_erfc}
\frac{e^{-x^2}}{\sqrt{\pi}}\left(\frac{1}{x} - \frac{1}{x^3}\right) \le \erfc(x) \le \frac{e^{-x^2}}{\sqrt{\pi}}\left(\frac{1}{x}\right)
\end{equation}
as one can easily infer by using the asymptotic expansion of the complementary error function.

\begin{lem}\label{calc}
Let $\varphi \colon \R \to (0,\infty)$ be the function defined as
\[
\varphi(s) :=  \frac{P_\gamma(H_s)}{\gamma(H_s)}. %= \frac{e^{-\frac{s^2}{2}}}{\phi(s)}.
\]
Then, $\varphi'(s)<0$ for all $s\in \mathbb{R}$, and $\lim_{s\to-\infty} \varphi(s)=+\infty$. %$\varphi'(s)<0$ for every $s$.% In particular, for every $R>0$ it holds 
%\[
%\max_{|s|\leq R} \varphi'(s) <0.
%\]
\end{lem}

\begin{proof}
The first part of the claim is equivalent to show that the function $f(s) =\sqrt{2\pi}(  \varphi(s))^{-1}$ satisfies $f'(s)>0$ for all $s \in \R$. 
Using the definition of the Gaussian perimeter and volume, we may equivalently write
\[
f(s) = \sqrt{2\pi} \, e^{\frac{s^2}{2}} \phi(s).
\]
As $\sqrt{2\pi} \, \phi'(s) = e^{-\frac{s^2}{2}}$, one readily computes the first derivative of $f$ 
\begin{equation}\label{eq:f'}
f'(s) = 1 +s f(s) = 1+se^{\frac{s^2}{2}} \sqrt{2\pi} \phi(s),
\end{equation}
which in particular is continuous. Trivially, $f'(s)\geq 1$  for $s\geq 0$. Thus we are left to check that $f'(s)>0$, for values $s <0$. \par

By integration by parts we have 
\[
\sqrt{2\pi} \phi(s) = \int_{-\infty}^s \left(-\frac{1}{t}\right) (-te^{-\frac{t^2}{2}})\, \textrm{d}t = - \frac{e^{-\frac{s^2}{2}}}{s} -  \int_{-\infty}^s\frac{1}{t^2} e^{-\frac{t^2}{2}}\, \textrm{d}t,
\]
which plugged in~\eqref{eq:f'} yields
\begin{equation*}
f'(s) = -s e^{\frac{s^2}{2}} \int_{-\infty}^s\frac{1}{t^2} e^{-\frac{t^2}{2}}\, \textrm{d}t.
\end{equation*}
It is immediate that, for any fixed $s<0$, this is positive. Hence, we have the first part of the claim.\par

To check the second part, we use the upper bound on $\erfc$ given in~\eqref{eq:Taylor_erfc}. For $s<< -1$, we have
\begin{equation*}
\frac{1}{\sqrt{2 \pi}}\varphi(s) = \frac{e^{-\frac{s^2}{2}}}{\int_{-\infty}^s e^{-\frac{t^2}{2}}\textrm{d}t} =  \frac{e^{-\frac{s^2}{2}}}{\int_{|s|}^{+\infty} e^{-\frac{t^2}{2}}\textrm{d}t} = \frac{e^{-\frac{s^2}{2}}}{\sqrt{\frac{\pi}{2}}\erfc\left(\frac{|s|}{\sqrt{2}} \right)} \ge |s|,
\end{equation*}
which completes the proof.
\end{proof}

For the sake of completeness, we remark two consequences of Lemma~\ref{calc}. First, Gaussian Cheeger sets exist. Second,  the Cheeger set of any given halfspace $H_s$ is the halfspace itself. Indeed, when proving existence one easily sees that any minimizing sequence $\{E_k\}_k$ is bounded in $BV_\gamma(\R^n)$ and hence, up to a subsequence, it converges to some set $E$. By the lower semicontinuity of $P_\gamma(\cdot)$, in order to show that this limit set is a minimizer one only needs to check that $\gamma(E)>0$. The previous lemma can be used to show that minimizing sequences are such that $\gamma(E_k)$ does not converge to $0$, as the ratio $P_\gamma(E_k)\gamma(E_k)^{-1}$ would otherwise be unbounded. Regarding the minimality of $H_s$, the Gaussian isoperimetric inequality ensures that for any fixed volume the halfspace is the perimeter minimizer, while the lemma ensures that any halfspace $H_\sigma$ strictly contained in $H_s$ has ratio $P_\gamma(H_\sigma)\gamma(H_\sigma)^{-1}$ bigger than  that of $H_s$.

\begin{lem}\label{lem:Phi}
Let $\Phi : [0,1] \to [0,1]$ be the function 
\begin{equation}\label{eq:def_Phi}
\Phi(\rho) = \frac{\rho}{1 + \sqrt{|\log(\rho)|}},
\end{equation}
defined by continuity at $\rho=0$ as $\Phi(0)=0$. Then, $\Phi$ is increasing, with 
\begin{equation}\label{eq:prop_phi_above}
\Phi(\rho) \le \rho
\end{equation}
and 
\begin{equation}\label{eq:prop_phi_below}
\Phi\left(\frac14 \rho\right) \geq \frac{1}{4(1+\sqrt{\log(4)})} \Phi(\rho).
\end{equation}
\end{lem}

\begin{proof}
We begin by showing that $\Phi$ is increasing. We have
\begin{equation*}
\Phi'(\rho) = \frac{1}{1 + \sqrt{|\log(\rho)|}} - \frac{\sgn(\log(\rho))}{2\sqrt{|\log(\rho)|}(1 + \sqrt{|\log(\rho)|})^2}.
%&=\frac{2\sqrt{|\log(\rho)|}(1 + \sqrt{|\log(\rho)|})-\sgn(\log(\rho))}{2\sqrt{|\log(\rho)|}(1 + \sqrt{|\log(\rho)|})^2}.
\end{equation*}
Since $\rho\in[0,1]$, we immediately get $\Phi'(\rho)\ge0$ as $\sgn(\log(\rho)) \leq 0$. Also the bound \eqref{eq:prop_phi_above} is trivial. 

We are left with \eqref{eq:prop_phi_below}. To this aim we write 
\begin{equation*}
\Phi(\rho/4)=\frac{\rho/4}{1+\sqrt{|\log(\rho/4)|}} = \frac 14 \Phi(\rho) \frac{1+\sqrt{|\log(\rho)|}}{1+\sqrt{|\log(\rho/4)|}}.
\end{equation*}
The claim follows as
\[
g(\rho):=\frac{1+\sqrt{|\log(\rho)|}}{1+\sqrt{|\log(\rho/4)|}}
\]
attains its minimum over the interval $[0,1]$ at  $\rho=1$. Indeed, one can check that $g(\rho)$ is decreasing in $[0,1]$. 
%Computing $g'(\rho)$, one sees that checking $g'(\rho)\le 0$ is equivalent to checking
%\[
%|\log \rho| - \left |\log \frac \rho4 \right| + \sqrt{|\log \rho|} - \sqrt{ \left|\log \frac \rho4\right|} \le 0,
%\]
%which is trivially true from the monotonic decreasing behavior of $|\log(\cdot)|$ and $\sqrt{|\log(\cdot)|}$ in $[0,1]$.
\end{proof}

Thanks to the monotonicity property stated in the above lemma, we can now show that we can control from above $\Phi(|b(E)|)$ with the mass of the set itself $\gamma(E)$, up to some multiplicative, universal constant.

\begin{lem}
\label{simple}
There is a constant $C$ such that for any set $E \subset \R^n$   it holds
\[
\Phi(|b(E)|) \leq C\gamma(E). 
\]
\end{lem}

\begin{proof}
First recall that for  any set $E$ it holds $|b(E)| \leq |b(H_E)|$. Therefore,  by the monotonicity of $\Phi$ we have
\[
\Phi(|b(E)|) \leq \Phi(|b(H_E)|).
\]
Hence, to prove the claim it suffices to show the validity of the inequality for halfspaces. Second, as $\Phi(|b(E)|) \le 1$, it is enough to prove the inequality for small values of $\gamma(H_E)$, and thus of $|b(H_E)|$. \par

We denote by $s_E:= \phi^{-1}(\gamma(H_E))$. To conclude we need to prove the inequality
\[
\Phi(|b(H_E)|) \leq  C\gamma(H_E),
\]
for values $s_E<0$ such that $|s_E|>>1$. First, notice that
\begin{align*}
\gamma(H_E) = \frac{1}{\sqrt{2 \pi}} \int_{-\infty}^{s_E} e^{-\frac{t^2}{2}} \, \textrm{d}t, && |b(H_E)| = \frac{1}{\sqrt{2\pi}}e^{-\frac{s_E^2}{2}}.
\end{align*}
On the one hand,  for all $s$ we have 
\begin{equation}\label{eq:stima1}
\begin{split}
\Phi(|b(H_s)|)&=\Phi\left( \frac{1}{\sqrt{2\pi}} e^{-\frac{s^2}{2}}\right)= \frac{e^{-\frac{s^2}{2}}}{\sqrt{2\pi}\left(1+\sqrt{\left|\log((2\pi)^{-\frac 12}e^{-\frac{s^2}{2}})\right| }\right)}\\
&= \frac{e^{-\frac{s^2}{2}}}{\sqrt{2\pi}\left(1+\sqrt{\frac 12\log(2\pi) + \frac{|s|^2}{2} }\right)} \le \frac{e^{-\frac{s^2}{2}}}{\sqrt{2\pi}\left(1+ \frac{|s|}{\sqrt 2}\right)}.
\end{split}
\end{equation}
On the other hand, for $s<0$ such that $|s|>>1$, using the asymptotic behavior of $\erfc$ given in~\eqref{eq:Taylor_erfc}, we have
\begin{equation}\label{eq:stima2}
\begin{split}
\gamma(H_s) &= \phi(s) = \frac{1}{\sqrt{2\pi}} \int_{-\infty}^s e^{-\frac{t^2}{2}}\textrm{d}t = \frac{1}{\sqrt{2\pi}} \int_{|s|}^{+\infty} e^{-\frac{t^2}{2}}\textrm{d}t\\
 &=  \frac{1}{\sqrt{2\pi}} \sqrt{\frac{\pi}{2}} \erfc\left( \frac{|s|}{\sqrt{2}}\right) \ge  \frac{e^{-\frac{|s|^2}{2}}}{\sqrt{2\pi}}   \left(\frac{1}{|s|} - \frac{2}{|s|^3} \right)
\end{split}
\end{equation}
Using~\eqref{eq:stima1} and~\eqref{eq:stima2} the claim boils down to check that there exists a constant $C>0$ such that the inequality
\[
C\left(1+ \frac{|s|}{\sqrt 2}\right) \ge  \frac{|s|^3}{|s|^2-2},
\]
holds for values $|s|>>1$.  This is obviously true for  $C\ge\sqrt{2}$.
\end{proof}

\subsection{Proof of Theorem~\ref{thm:gauss_alpha}}\label{ssec:gauss_alpha}

%Throughout the proof of the theorem we shall use several constants. For the sake of simplicity we shall denote them only by $c$, $c_1$ and $c_2$, and they can change from line to line.\par

Let $E \subset \Omega$ be a Cheeger set of $\Omega$. Then
\begin{equation} \label{ineq1}
h_\gamma(\Omega) - h_\gamma(H_\Omega) = \frac{P_\gamma(E)- P_\gamma(H_E)}{\gamma(E)} +\left(  \frac{P_\gamma(H_E)}{\gamma(E)}- \frac{P_\gamma(H_\Omega)}{\gamma(\Omega)}\right).
\end{equation}
Let us first show that we may assume that 
\begin{equation}\label{eq:e_omega}
\epsilon \leq \gamma(E)  \leq 1 -\epsilon,
\end{equation}
for  $\epsilon >0$ which depends on  $\gamma(\Omega)$. The upper bound in \eqref{eq:e_omega} follows simply from $E \subset \Omega$ and thus $\gamma(E) \leq \gamma(\Omega) \leq 1 - \epsilon_1$. Here $\epsilon_1 = 1 - \gamma(\Omega) >0$ as we may obviously assume that $\Omega \neq \R^n$.  
For the lower bound,  we first notice that $\alpha_\gamma(\Omega) \le 2\gamma(\Omega)$, and thus~\eqref{eq:che_gaus_alpha} immediately follows if $h_\gamma(\Omega) - h_\gamma(H_\Omega)\ge 2$, by choosing $c(\gamma(\Omega))\le \gamma(\Omega)^{-1}$. Hence, we may  assume $h_\gamma(\Omega) - h_\gamma(H_\Omega)< 2$.  Let us set $s_\Omega := \phi^{-1}(\gamma(\Omega))$ and $s_E := \phi^{-1}(\gamma(E))$. Clearly, $s_E<s_\Omega$. Then, we have 
\[
\varphi(s_E)-\varphi(s_\Omega)= \frac{P_\gamma(H_E)}{\gamma(E)}- \frac{P_\gamma(H_\Omega)}{\gamma(\Omega)} =  h_\gamma(\Omega) - h_\gamma(H_\Omega)< 2.
\]
The behavior of $\varphi(s) \to \infty $ as $s\to -\infty$ given in Lemma~\ref{calc} implies that $\gamma(E)>\epsilon_2=\e_2(\gamma(\Omega))$. Setting $\epsilon:=\min\{\epsilon_1, \epsilon_2\}$ we obtain \eqref{eq:e_omega}. In particular,  there exists $R=R(\gamma(\Omega))$ such that $-R \leq s_E \leq s_\Omega \leq R$.\par

By the quantitative isoperimetric inequality~\eqref{eq:gauss_iso_alpha} we have
\begin{equation} \label{ineq2}
\frac{P_\gamma(E)- P_\gamma(H_E)}{\gamma(E)} \geq \frac{1}{\gamma(E)}c(s_E)\alpha_\gamma^2(E),
\end{equation}
where $c(s_E) = \frac{1}{\kappa}(1+s_E^2)^{-1}e^{\frac{s_E^2}{2}}$ with $\kappa>0$ a universal constant, as can be seen in~\cite{BBJ17}. As $s_E\in [-R, R]$ we may  replace $\gamma(E)^{-1}c(s_E)$ with a constant $c=c(R)$, thus ultimately with  $c_1=c_1(\gamma(\Omega))$. \par

On the other hand, Lemma~\ref{calc} coupled with the fact that $(\phi^{-1})'>0$ and two subsequent applications of the Mean Value Theorem imply
\begin{equation}\label{ineq3}
\begin{split}
\frac{P_\gamma(H_E)}{\gamma(E)} - &\frac{P_\gamma(H_\Omega)}{\gamma(\Omega)} = -\varphi'(\xi) (s_\Omega - s_E) \ge \tilde c_2 (s_\Omega - s_E)\\
&\ge \tilde c_2 (\phi^{-1})'(\hat \xi) (\gamma(\Omega)- \gamma(E)) \ge  c_2 (\gamma(\Omega)- \gamma(E)),
\end{split}
\end{equation}
where the two constants can be chosen as $\tilde c_2= \min_{[-R,R]}\{-\varphi'(s)\}$, and  $c_2 =\tilde c_2 \,  \min_{[\epsilon, 1-\epsilon]}  \{(\phi^{-1})'(\rho)\}$, and thus they ultimately depend only on $\gamma(\Omega)$. Combining the  inequalities~\eqref{ineq1},~\eqref{ineq2} and~\eqref{ineq3} yields
\begin{align}
h_\gamma(\Omega) - h_\gamma(H_\Omega) &\geq c_1 \alpha_\gamma^2(E) + c_2 (\gamma(\Omega)- \gamma(E)) \nonumber \\
&\geq \frac 15 \min\{c_1, c_2\} ( \alpha_\gamma(E)  + \gamma(\Omega)- \gamma(E))^2\nonumber \\
&\geq c \, (\gamma(\Omega)) ( \alpha_\gamma(E)  + \gamma(\Omega)- \gamma(E))^2, \label{ineq4}
\end{align}
where we used that $\gamma(\Omega)-\gamma(E)\le 1$ and  $\alpha_\gamma(E)\le2$. \par

We now set $H_E$ to be the halfspace with the  same measure as $E$ such that $\alpha_\gamma(E) = \gamma(E \Delta H_{E})$, and $H_\Omega$ to be the halfspace with the same measure as $\Omega$ containing $H_E$. Then, we have
\[
\alpha_\gamma(E) = \gamma(E \Delta H_{E}) \geq  \gamma(E \setminus  H_{E}) = \gamma(E) - \gamma(E \cap  H_{E}).
\]
Using the inequality above, that $E \subset \Omega$, that $H_E \subset H_\Omega$ and the equality $2\gamma(\Omega \setminus H_{\Omega}) = \gamma(\Omega \Delta H_{\Omega})$ we get the following estimate
\begin{align}
 ( \alpha_\gamma(E)  &+ \gamma(\Omega)- \gamma(E))^2 \geq   (\gamma(\Omega) -  \gamma(E \cap  H_{E}))^2\nonumber \\
& \geq  (\gamma(\Omega) -  \gamma(\Omega \cap  H_{E}))^2 = \gamma(\Omega \setminus H_{E})^2 \nonumber \\
& \geq \gamma(\Omega \setminus H_{\Omega})^2  =\frac14 \gamma(\Omega \Delta H_{\Omega})^2 \geq \frac14 \alpha_\gamma^2(\Omega). \label{ineq5}
\end{align}
Inequality~\eqref{ineq4} paired with~\eqref{ineq5} finally yields the claim. 

\subsection{Proof of Theorem~\ref{thm:gauss_logbeta}}\label{ssec:gauss_logbeta}

Recall that $\beta_\gamma(\Omega) \le 1$, and recall the definition of $\Phi(\rho)$ given in~\eqref{eq:def_Phi}. Thus, we aim to show 
\[
h_\gamma(\Omega) - h_\gamma(H_\Omega) \geq c \, \Phi(\beta_\gamma(\Omega)).
\]
We begin by noticing that 
\[
\beta_\gamma(\Omega) =|b(H_\Omega)-b(\Omega)| \le 4|b(H_\Omega)|.
\]
Hence, by the properties of the function $\Phi$ established in Lemma~\ref{lem:Phi}, together with Lemma~\ref{simple}, we have
\[
\Phi(\beta_\gamma(\Omega)) \le C\gamma(\Omega).
\]
Thus, as in the proof of Theorem~\ref{thm:gauss_alpha} we may assume without loss of generality that $h_\gamma(\Omega)-h_\gamma(H_\Omega)<C$, since otherwise the claim follows as before. Then, we may argue again as in the proof of Theorem~\ref{thm:gauss_alpha} and we may assume without loss of generality
\[
\epsilon \leq \gamma(E)  \leq 1 -\epsilon,
\]
for  $\epsilon >0$ which depends only on $\gamma(\Omega)$, because otherwise the claim immediately follows. Then, we may argue again as in the proof of Theorem~\ref{thm:gauss_alpha} and obtain inequality~\eqref{ineq3}. In place of using the quantitative isoperimetric inequality~\eqref{eq:gauss_iso_alpha}  and obtaining~\eqref{ineq2} we use the strong quantitative isoperimetric inequality~\eqref{eq:gauss_iso_beta} and get
\begin{equation}\label{ineq6}
\frac{P_\gamma(E)- P_\gamma(H_E)}{\gamma(E)} \geq  \frac{1}{\gamma(E)} c(s_E)\beta_\gamma(E),
\end{equation}
now with $c(s_E)=\frac 1\kappa(1+s_E^2)^{-1}$ (see again~\cite{BBJ17}), and thus ultimately a constant $c_1=c_1(\gamma(\Omega))$.

Combining~\eqref{ineq1} with~\eqref{ineq6} and~\eqref{ineq3} we deduce
\begin{equation} \label{ineq7}
h_\gamma(\Omega) - h_\gamma(H_\Omega) \geq c_1 \beta_\gamma(E) + c_2 \,( \gamma(\Omega)- \gamma(E) ). 
\end{equation}
We fix $H_E$ to be the halfspace such that $\beta_\gamma(E)=|b(H_E)-b(E)|$, and we let $H_\Omega$ be the halfspace with the same measure of $\Omega$ such that $H_E \subset H_\Omega$. By splitting $H_\Omega$ into $H_E$ and  $H_\Omega\setminus H_E$, adding and removing $b(E)$, and using the triangle inequality,  we get
\begin{align*}
\beta_\gamma(\Omega) \le |b(H_\Omega) - b(\Omega)| &=  |b(H_E) - b(E) + b(H_\Omega\setminus H_E) + b(E) - b(\Omega)|\\
&\le \beta_\gamma(E) + |b(H_\Omega \setminus H_E)| + |b(E)-b(\Omega)|\\
&= \beta_\gamma(E) + |b(H_\Omega \setminus H_E)| + |b(\Omega \setminus E)| \\
&\le 2\max\{ \beta_\gamma(E), |b(H_\Omega \setminus H_E)| + |b(\Omega \setminus E)|\}.
\end{align*}
Thus we have either
\begin{equation}\label{ineq8}
\beta_\gamma(\Omega) \leq 2 \beta_\gamma(E),
\end{equation}
or
\begin{equation}\label{ineq9}
\beta_\gamma(\Omega) \le 2\big( |b(H_\Omega \setminus H_E)| + |b(\Omega \setminus E)|\big).
\end{equation}

If~\eqref{ineq8} holds true, we obtain the result by~\eqref{eq:prop_phi_above}, \eqref{ineq8} and~\eqref{ineq7}. If~\eqref{ineq9} holds true, we have that
\[
\frac 14 \beta_\gamma(\Omega) \le \max\{|b(H_\Omega \setminus H_E)|, |b(\Omega \setminus E)|\}.
\]
By the monotonicity of $\Phi$ and its property~\eqref{eq:prop_phi_below}, it follows
\[
\Phi(\beta_\gamma(\Omega)) \leq 4(1+\sqrt{\log(4)}) \max\{\Phi(|b(H_\Omega \setminus H_E)|), \Phi(|b(\Omega \setminus E)|)\}.
\]
As the sets $H_\Omega \setminus H_E$ and $\Omega \setminus E$ have the same Gaussian measure, and as $E\subset \Omega$, by Lemma~\ref{simple} we finally get
\[
\Phi(\beta_\gamma(\Omega)) \leq C (\gamma(\Omega \setminus E)) = C(\gamma(\Omega)-\gamma(E)),
\]
which coupled with~\eqref{ineq7} allows us to conclude.

\subsection{The sharpness of the inequality with the index $\beta_\gamma$}\label{ssec:gaus_beta_not}

In this section, we show that the dependence on the asymmetry in Theorem \ref{thm:gauss_logbeta} is sharp.   Let $T>>1$, and let us define the family of one-dimensional sets $\{\Omega_T\}_T$ as
\[
\Omega_T := (-\infty, -1)\cup(T,+\infty).
\]

%that there is no constant $c=c(\gamma(\Omega))$ such that inequality
%\begin{equation}\label{eq:failure_beta_gauss}
%h_\gamma(\Omega) - h_\gamma(H_\Omega)\ge c\beta_\gamma(\Omega),
%\end{equation}
%holds true, by building a suitable family of sets. The construction also

 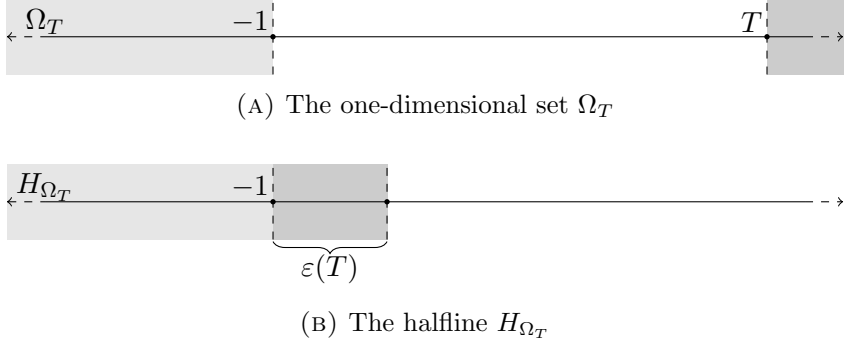
\begin{figure}[t]
%%%%FIG.1
   	 \begin{subfigure}{1\linewidth}
\centering
\begin{tikzpicture}
\fill[fill=black!10!white] (-5.5,-.5)  rectangle (-2,.5);
\fill[fill=black!20!white] (4.5,-.5)  rectangle (5.5,.5);
%\fill[fill=black!20!white] (-1,-1.5)  rectangle (-.2,1.5);
\draw (-5,.2) node {$\Omega_T$};
%\draw (0,.2) node {$0$};
%\fill (0,0) circle (1pt);
%\draw (-6,1) node {$H_\Omega$};
\draw (-2.3,0.2) node {$-1$};
\fill (-2,0) circle (1pt);
%\draw [decorate,decoration={brace,amplitude=2pt,mirror,raise=4ex}]
%  (-1,0) -- (-.2,0) node[midway,yshift=-3em]{$\epsilon(T)$};
\draw (4.3,0.2) node {$T$};
\fill (4.5,0) circle (1pt);
\draw[black, dashed] (-2, -.5) -- (-2, .5);
\draw [black, dashed] (4.5,-.5) -- (4.5,.5);
%\draw[black, dashed] (-.2, -1.5) -- (-.2, 1.5);
\draw (-5,0) -- (5,0); 
\draw[<-, dashed] (-5.5, 0) -- (-5,0);
\draw[->, dashed] (5, 0) -- (5.5,0);
%\draw[<->] (0,-1.5) -- (0,1.5);	
\end{tikzpicture}
\caption{The one-dimensional set $\Omega_T$}
\label{fig:ag}
	\end{subfigure}

\vskip .5cm

%%%%FIG.2
 	  \begin{subfigure}{1\linewidth}
\centering
\begin{tikzpicture}
\fill[fill=black!10!white] (-5.5,-.5)  rectangle (-2,.5);
%\fill[fill=black!20!white] (5.5,-1.5)  rectangle (6.5,1.5);
\fill[fill=black!20!white] (-2,-.5)  rectangle (-.5,.5);
%\draw (-6,1) node {$\Omega$};
\draw (-5,.2) node {$H_{\Omega_T}$};
%\draw (0,.2) node {$0$};
%\fill (0,0) circle (1pt);
\draw (-2.3,0.2) node {$-1$};
\fill (-2,0) circle (1pt);
\draw [decorate,decoration={brace,amplitude=4.5pt,mirror,raise=15}]
  (-2,0) -- (-.5,0) node[midway,yshift=-2.3em]{$\e(T)$};
\fill (-.5,0) circle (1pt);
%\draw (5.3,0.2) node {$T$};
\draw[black, dashed] (-2, -.5) -- (-2, .5);
%\draw [black, dashed] (5.5,-1.5) -- (5.5,1.5);
\draw[black, dashed] (-.5, -.5) -- (-.5, .5);
\draw (-5,0) -- (5,0); 
\draw[<-, dashed] (-5.5, 0) -- (-5,0);
\draw[->, dashed] (5, 0) -- (5.5,0);%\draw[<->] (0,-1.5) -- (0,1.5);	
\end{tikzpicture}
\caption{The halfline $H_{\Omega_T}$}
\label{fig:bg}
  	\end{subfigure}
    
    \caption{The set $\Omega_T$ and the corresponding halfline $H_{\Omega_T}$.}
    \label{fig:Homegas}
    \end{figure}
    
It is easy to notice, that for $T>1$ the Cheeger set of $\Omega$ is given by the halfline $(-\infty, -1)$. The halfline of same volume of $\Omega$ is
\[
H_{\Omega_T} = (-\infty, -1+\epsilon(T)),
\]
where $\epsilon(T)$ is such that
\begin{equation}\label{eq:e(T)}
 \int_T^{+\infty}e^{-\frac{t^2}{2}}\mathrm{d}t = \int_{-1}^{-1+\epsilon(T)}e^{-\frac{t^2}{2}}\mathrm{d}t,
\end{equation}
and obviously $\epsilon(T)\to 0$, as $T\to +\infty$. For the convenience  of the reader, $\Omega_T$ is depicted in Figure~\ref{fig:ag}, while the corresponding halfline $H_{\Omega_T}$ in Figure~\ref{fig:bg}. By recalling the definition of the function $\varphi$ in Lemma \ref{calc}, we have by the Mean Value Theorem
\[
h_\gamma(\Omega) - h_\gamma(H_\Omega) = \frac{P_\gamma((-\infty, -1))}{\gamma((-\infty, -1))} - \frac{P_\gamma((-\infty, -1+\epsilon(T)))}{\gamma((-\infty, -1+\epsilon(T)))} = -\varphi'(\xi)\epsilon(T),
\]
for some $\xi \in (-1, -1+\epsilon(T))$. Hence, we immediately obtain the upper bound
\begin{equation}\label{eq:failure_gauss_upper}
h_\gamma(\Omega) - h_\gamma(H_\Omega) \le C\, \epsilon(T),
\end{equation}
by choosing $T>>1$ in such a way that $\epsilon(T)<1$ and $C=\min_{[-1,0]}\{-\varphi'(x)\}$ independent of $\epsilon$. We now aim to bound $\beta_\gamma(\Omega)$ from below and show that
\begin{equation}\label{eq:comb0}
\frac{\beta_\gamma(\Omega)}{1+ \sqrt{|\log(\beta_\gamma(\Omega))|}} \geq c \, \epsilon(T).
\end{equation}
This inequality and inequality~\eqref{eq:failure_gauss_upper} immediately show that Theorem~\ref{thm:gauss_logbeta} is sharp. Let us prove inequality~\eqref{eq:comb0}.\par 

By definition of $\beta_\gamma(\Omega)$, we have
\begin{equation}\label{eq:comb1}
\beta_\gamma(\Omega) = |b(H_\Omega)-b(\Omega)| %&= \left|\int_{-\infty}^{-1+\epsilon(T)} te^{-\frac{t^2}{2}}\mathrm{d}t\right| -  \left|\int_{-\infty}^{-1} te^{-\frac{t^2}{2}}\mathrm{d}t + \int_{T}^{+\infty} te^{-\frac{t^2}{2}}\mathrm{d}t \right|\\
= \frac{1}{\sqrt{2\pi}}\left(e^{-\frac{(\epsilon(T)-1)^2}{2}} -e^{-\frac 12} + e^{-\frac{T^2}{2}}\right).
\end{equation}
We now use~\eqref{eq:e(T)} to bound this quantity from below as follows. On the one hand, for the LHS of~\eqref{eq:e(T)}, one has
\begin{equation*}
\int_T^{+\infty}e^{-\frac{t^2}{2}}\mathrm{d}t = \sqrt{2}\int_{\frac{T}{\sqrt{2}}}^{+\infty} e^{-x^2}\mathrm{d}x = \sqrt{\frac{\pi}{2}} \erfc\left(\frac{T}{\sqrt{2}}\right) \le \frac 1T e^{-\frac{T^2}{2}} ,
\end{equation*}
where we used the asymptotic behavior of $\erfc$ given in~\eqref{eq:Taylor_erfc}. On the other hand, for the RHS of~\eqref{eq:e(T)}, one has
\[
\int_{-1}^{-1+\epsilon(T)}e^{-\frac{t^2}{2}}\mathrm{d}t \ge \int_{-1}^{-1+\epsilon(T)} -te^{-\frac{t^2}{2}}\mathrm{d}t = e^{-\frac{t^2}{2}}\Big|_{-1}^{-1+\epsilon(T)} = e^{-\frac{(\epsilon(T)-1)^2}{2}} -e^{-\frac 12}.
\]
Combining these two inequalities, we get
\begin{equation}\label{eq:comb2}
e^{-\frac{T^2}{2}} \ge  T\left(e^{-\frac{(\epsilon(T)-1)^2}{2}} -e^{-\frac 12}\right).
\end{equation}
Thus by~\eqref{eq:comb1} and~\eqref{eq:comb2} we get
\begin{equation}\label{eq:comb3}
\beta_\gamma(\Omega) \ge \frac{1}{\sqrt{2\pi}}(1+T)\left(e^{-\frac{(\epsilon(T)-1)^2}{2}} -e^{-\frac 12}\right) \geq c\, (1+T)\, \epsilon(T) .
\end{equation}

By using again the asymptotic behavior of $\erfc$ given in~\eqref{eq:Taylor_erfc} we may estimate the LHS of~\eqref{eq:e(T)} as
\[
\int_T^{+\infty}e^{-\frac{t^2}{2}}\mathrm{d}t \geq \frac{1}{2T} e^{-\frac{T^2}{2}},
\]
and we simply estimate the RHS of~\eqref{eq:e(T)} as
\[
\int_{-1}^{-1+\epsilon(T)}e^{-\frac{t^2}{2}}\mathrm{d}t \leq 2 \int_{-1}^{-1+\epsilon(T)} -te^{-\frac{t^2}{2}}\mathrm{d}t =2 \left(e^{-\frac{(\epsilon(T)-1)^2}{2}} -e^{-\frac 12}\right)
\]
for  $T>>1$. Hence, we deduce by~\eqref{eq:e(T)}
\[
e^{-\frac{T^2}{2}} \leq 4  T\left(e^{-\frac{(\epsilon(T)-1)^2}{2}} -e^{-\frac 12}\right) \leq c\, T\, \epsilon(T)
\]
for $T>>1$. Combining this last inequality with inequality~\eqref{eq:comb3} yields
\begin{equation}\label{eq:comb4}
\beta_\gamma(\Omega)\ge c(1+T)\e(T) \ge c T\e(T) \ge c \, e^{-\frac{T^2}{2}} \ge e^{-T^2},
\end{equation}
for T big enough. Recalling that $\beta_\gamma(\Omega)<1$ for $T>>1$ by~\eqref{eq:comb4} we get
\[
\sqrt{|\log (\beta_\g(\Omega))|}\le T,
\]
for $T$ big enough. This and inequality~\eqref{eq:comb3} imply the inequality~\eqref{eq:comb0}. \par

Counterexamples in higher dimensions can be constructed in the same exact way. Notice that in higher dimensions, one can as well provide P-connected counterexamples, by adding a suitably thin tube connecting the two halfspaces defining $\Omega_T$.

\section*{Acknowledgements}

The present research was carried out during a visit of G.~S.~at the University of Jyv\"askyl\"a.  G.~S.~wishes to thank the hosting institution for the kind hospitality and the INdAM institute of which he is a member and which funded his stay in Jyv\"askyl\"a (grant n.~U-UFMBAZ-2018-000928 03-08-2018). G.~S.~was also partially supported by the INdAM--GNAMPA 2019 project ``Problemi isoperimetrici in spazi Euclidei e non'' (grant n.~U-UFMBAZ-2019-000473 11-03-2019).  V.~J.~was supported by the Academy of Finland grant 314227.   
    
\bibliographystyle{plainurl}

\bibliography{Cheeger_quantitative}

\end{document}